\documentclass{alggeom}
\usepackage{amsmath}
\usepackage{color}
\usepackage{MnSymbol}

\usepackage{ifthen}
\newcommand{\cit}[1]{{\rm \textbf{#1}}}

\newcommand{\Ref}[2]{\cit{%
\ifthenelse{\equal{#1}{thm}}{Theorem}{}%
\ifthenelse{\equal{#1}{prop}}{Proposition}{}%
\ifthenelse{\equal{#1}{lem}}{Lemma}{}%
\ifthenelse{\equal{#1}{cor}}{Corollary}{}%
\ifthenelse{\equal{#1}{defn}}{Definition}{}%
\ifthenelse{\equal{#1}{oss}}{Remark}{}%
\ifthenelse{\equal{#1}{sec}}{Section}{}%
\ifthenelse{\equal{#1}{subsec}}{Subsection}{}%
\ifthenelse{\equal{#1}{ex}}{Example}{}%
\ifthenelse{\equal{#1}{conj}}{Conjecture}{}%
\ifthenelse{\equal{#1}{ssec}}{Subsection}{}%
\ifthenelse{\equal{#1}{tab}}{Table}{}%
\ifthenelse{\equal{#1}{cla}}{Claim}{}%
\  \ref{#1:#2}%
}}

\theoremstyle{plain} 
\newtheorem{prop}{Proposition}[section]
\newtheorem{thm}[prop]{Theorem}
\newtheorem{lem}[prop]{Lemma} 
\newtheorem{cor}[prop]{Corollary}

\theoremstyle{remark}

\newtheorem{oss}[prop]{Remark}
\newtheorem{ex}[prop]{Example}

\theoremstyle{definition}

\newcommand{\hk}{hyperk\"{a}hler }

\newcommand{\kahl}{K\"{a}hler }

\newcommand{\ktipo}{$K3^{[2]}$ type}
\newcommand{\kntipo}{$K3^{[n]}$ type}
\newcommand{\kntiposp}{$K3^{[n]}$ type }

\DeclareMathOperator*{\rk}{rk}

\newcommand{\NS}{\mathrm{NS}}

\newcommand{\Z}{\mathbb{Z}}
\newcommand{\C}{\mathbb{C}}
\newcommand{\Q}{\mathbb{Q}}

\newcommand{\PP}{\mathbb{P}}

\let\div\undefined
\DeclareMathOperator{\div}{div\!}
\newcommand{\Ker}{\mathrm{Ker}}
\newcommand{\aut}{\mathrm{Aut}}
\newcommand{\id}{\mathrm{id}}

\renewcommand{\mod}[1]{\mathrm{\text{ }(mod\text{ }#1)}}

\newcommand{\lat}[1]{\langle#1\rangle}
\usepackage{hyperref}
\newcommand\numberthis{\addtocounter{equation}{1}\tag{\theequation}}

\begin{document}

\title{Automorphisms of generalised Kummer fourfolds}
\author{Giovanni Mongardi}
\email{giovanni.mongardi@unimi.it}
\address{Department of Mathematics, University of Milan\\ via Cesare Saldini 50, Milan, Italy}
\author{K\'evin Tari}
\email{kevin.tari@math.univ-poitiers.fr}
\address{Laboratoire de Math\'ematiques et Applications, Universit\'e de Poitiers\\
T\'el\'eport 2 - BP 30179
Boulevard Marie et Pierre Curie\\
86962 Futuroscope Chasseneuil Cedex,
France }
\author{Malte Wandel}
\email{wandel@kurims.kyoto-u.ac.jp}
\address{RIMS, Kyoto University\\ Kitashirakawa Oiwake cho, Sakyo-ku,
Kyoto 606-8502, Japan}

\begin{abstract}
We classify non symplectic prime order automorphisms and all finite order symplectic automorphism groups of generalised Kummer fourfolds using lattice theory and recent results on ample cones and monodromy groups. We study various geometric realisations of the automorphisms obtained in the classification. In the case of higher dimensional generalised Kummers, we provide full results for finite symplectic automorphisms groups and we sketch the classification for prime order non-symplectic groups.
\end{abstract}
\thanks{First named author was supported by FIRB 2012 ``Spazi di Moduli e applicazioni''\\Third named author is supported by JSPS Grant-in-Aid for Scientific Research (S)25220701}

\maketitle
\tableofcontents
\setcounter{section}{-1}
\section{Introduction}
In recent years a lot of attention was drawn to the study of hyperk\"ahler manifolds and their automorphisms. After Beauville's (\cite{Beau83} and Huybrechts' (\cite{Huy99}) fundamental results, most of these developments studied various aspects of automorphisms of Hilbert schemes of points on K3 surfaces and their deformations (\cite{BCMS14}, \cite{BCNS14}, \cite{BCS14}, \cite{BNS13}, \cite{BS12}, \cite{Boi12}, \cite{Cam12}, \cite{Mon13}, \cite{OW13}). The other famous series of deformation classes -- the so-called \em generalised Kummer manifolds \em or \em hyperk\"ahler manifolds of Kummer type \em -- has been only been treated in very few articles: In \cite{BNS11} the authors determined the kernel of the cohomological representation, which associates to every automorphism of a hyperk\"ahler manifold the induced action on the second integral cohomology lattice. Oguiso (\cite{Ogi12}) further studied the automorphisms in this kernel, proving that they, in fact, act non-trivially on the total cohomology of the underlying manifolds. In \cite{OS11} examples of fixed point free automorphisms on generalised Kummer manifolds are presented. In \cite{MW14}, the so-called \em induced automorphisms \em have been introduced. This construction starts with a group homomorphism of a 2-torus, which then, under certain conditions, induces an automorphism of the Albanese fibre of certain moduli spaces of stable objects on the torus, which are hyperk\"ahler manifolds of Kummer type.

One of the first tasks concerning classifications of automorphisms of hyperk\"ahler manifolds is the classification of prime order automorphisms. Such an automorphism is either \em symplectic \em, i.e.\ preserving the symplectic structure of the manifold, or \em non-symplectic \em otherwise. After dividing prime order automorphisms into these two groups, the classification is usually done using lattice theory. The case of K3 surfaces has been done by \cite{AST11} for the non-symplectic case and \cite{GS07} in the symplectic case. In the case of K3$^{[2]}$-type manifolds we have the results of \cite{BCS14} and \cite{MW14} in the non-symplectic case and \cite{Mon14} in the symplectic case.

The main aim of this article is to provide a first step towards a classification of automorphisms of prime order in the case of generalised Kummer fourfolds. We give a lattice theoretic classification in both, the symplectic and the non-symplectic case, using recent developments concerning ample cones and monodromy groups of generalised Kummer manifolds (cf.\ \cite{Mon13b}). Furthermore we analyse the lattice theoretic result by studying the corresponding geometric realisations if available. We use the survey \cite{MTW15} as a main source to study natural and induced automorphisms on Kummer fourfolds.

\subsection*{Acknowledgements} 
We would like to thank the Max Planck Institute in Mathematics, for hosting the three authors when this work was started. We are also grateful to Samuel Boissi\`ere and Bert Van Geemen for their comments and support. The first named author would also like to thank G. H\"ohn for his many insightful comments on this paper.

\section{Preliminaries}\label{sec:prelim}
In this first section, we gather the required background material and fix some notation.

\subsection{Generalised Kummer fourfolds}
Let $A$ be a complex 2-torus and denote by $A^{[3]}$ the Hilbert scheme of three points on $A$. The fibres of the isotrivial Albanese map $A^{[3]}\rightarrow A$ are hyperk\"ahler manifolds of dimension four known as \em generalised Kummer fourfolds. \em We call \em fourfolds of Kummer type \em all hyperk\"ahler deformations of generalised Kummer fourfolds.

Let $X$ be a fourfold of Kummer type, then $H^2(X,\mathbb{Z})=U^3\oplus \lat{-6}$, where the pairing on $H^2$ is the Beauville-Bogomolov-Fujiki form.

By \cite{BNS11}, the cohomological representation
\[\nu\colon \aut(X)\rightarrow O(H^2(X,\Z))\]
has kernel isomorphic to $A[3]\rtimes\pm\id_A$. We introduce the following convention: An automorphism of $X$ is said to be an \em automorphism of order $n$ \em if the induced action on $H^2$ has order $n$.

The discriminant group of $H^2(X,\mathbb{Z})$ is isomorphic to $\Z/6\Z$ and will be denoted by $A_X$. For any lattice $L$ and any element $v\in L$ we denote by $\div_L(v)$ the divisibility of $v$ in $L$, i.e.\ the positive generator of the ideal $(v,L)$.

\subsection{Monodromy and ample cone}
In order to study automorphisms, it is vital to understand the birational geometry of our manifolds. The latter is governed by \em parallel transport operators \em or \em monodromy operators, \em i.e.\ isometries of $H^2$ which are induced by parallel transport along a connected base. We have the following Hodge-theoretic Torelli Theorem due to Huybrechts, Markman and Verbitsky.

\begin{thm}[{\cite[Thm.\ 1.3]{Mar09}}]\label{thm:torelli}
Let $X$ and $X'$ be hyperk\"ahler manifolds. Then $X$ and $X'$ are birational if and only if there is a parallel transport operator $H^2(X',\Z)\rightarrow H^2(X,\Z)$ which is a Hodge isometry.
\end{thm}

The group of parallel transport operators (denoted $Mon^2(X)$) for Kummer fourfolds has been computed by Markman and Mehrotra \cite[Corollary 4.8]{MM12}:

\begin{prop}[{[Markman-Mehrotra]}]\label{prop:mono_kum}
Let $\mathcal{W}(X)$ be the group of orientation preserving isometries of $H^2(X,\mathbb{Z})$. Let $\mathcal{N}(X)$ be the kernel of the map $\mathrm{det}\circ\chi\colon\mathcal{W}(X)\rightarrow {\pm 1}$, where $\chi$ is the character of the action on $A_X$. Then $Mon^2(X)=\mathcal{N}(X)$
\end{prop}

We call a \em marking \em of a Kummer fourfold $X$ an isometry of $H^2(X,\Z)$ with the reference lattice $U^3\oplus\langle-6\rangle$. A \em marked pair \em consists of a manifold and a marking. An \em isomorphism of marked pairs \em is an isomorphism respecting the markings. The above result implies that the moduli space of marked generalised Kummer fourfolds has four connected components. In a certain sense this is the smallest possible. Indeed, the moduli space of marked 2-tori also has four connected components (cf.\ \cite[Section 2]{MTW15}).

Within the set of birational maps, automorphism are exactly those which map the ample cone of a manifold to itself. Thus, the understanding of the ample cone is crucial for the study of automorphisms of any manifold. The ample cone for Kummer type fourfolds arising from moduli spaces of sheaves on abelian surfaces has been studied by Yoshioka \cite{Yos12}. His results can be generalised to any manifold of Kummer type, using either \cite{BHT13} or \cite{Mon13b}. In particular, there are three kinds of faces of the ample cone, each of them orthogonal to one of the following:
\begin{itemize}
\item Divisors $\delta$ of square $-6$ and divisibility $6$.
\item Divisors $\delta$ of square $-6$ and divisibility $3$.
\item Divisors $\delta$ of square $-6$ and divisibility $2$.
\end{itemize}
In the following, divisors of these kinds are called wall divisors. 
The first kind of divisors is not effective, however $2\delta$ (or $-2\delta$) is a sum of reduced and irreducible uniruled divisors which can be contracted to a symplectic surface.
The second kind of divisors is effective (or its opposite is effective) and again is a sum of reduced and irreducible uniruled divisors which can be contracted to a symplectic surface.
The last kind of divisors are not effective and none of their multiples is. However they correspond to $\mathbb{P}^2$'s inside our fourfold which induce Mukai flops. (In particular, $\delta/2$ is a class in $H_2(X,\mathbb{Z})$ which is the class of a line in such a $\mathbb{P}^2$.) 

\subsection{Moduli spaces}
Besides generalised Kummer manifolds there is up to now only one other construction of manifolds of Kummer type: Let $A$ be an abelian surface and $v=(r,l,s)\in H^0(A)\oplus \NS(A)\oplus H^4(A)$ a so-called \em Mukai vector \em satisfying $r\geq 0,$ $l$ effective for $r=0$ and $v^2:=l^2-2rs=6.$ Assume that there exists a $\varphi$-invariant $v$-generic stability condition. Then the Albanese fibre of the associated moduli space of stable objects $M(v)$ is a fourfold of Kummer type.

In \cite{MW14} the authors gave a lattice theoretic criterion to identify moduli spaces of stable objects.

\begin{prop}[{\cite[Prop.\ 2.4]{MW14}}]\label{prop:mod_spac}
Let $X$ be a fourfold of Kummer type. If the algebraic part of the Hodge structure on $\Lambda$ (which is induced by the embedding $H^2(X,\Z)\hookrightarrow \Lambda$) contains a copy of $U$ as a direct summand, then $X$ is the Albanese fibre of a moduli space of stable objects.
\end{prop}

\subsection{Lattices}\label{subsec:lat}
Fix a Kummer fourfold $X$ and let $G$ be a finite group of automorphisms. Every element of $G$ acts on the holomorphic two-form of $X$ by homotheties. We call an automorphism \em symplectic \em if it preserves the symplectic form and \em non-symplectic \em otherwise. We denote by $T_G(X):=H^2(X,\mathbb{Z})^G$ the \em invariant lattice \em and by $S_G(X):=T_G(X)^\perp$ the \em coinvariant lattice.\em

The fact that $H^2(X,\Z)$ is not unimodular often makes things more complicated. It is therefore sometimes convenient to switch to a bigger unimodular lattice: We consider a primitive embedding $H^2(X,\mathbb{Z})\hookrightarrow \Lambda:=U^4$ -- such an embedding is unique up to an isometry of $\Lambda$ -- that sends $U^3\subset H^2(X)$ identically into the first three copies of $U$ and sends $\delta$, a generator of $(U^3)^\perp$, to $e-3f$, where $e,f$ form a typical basis of the last copy of $U$. The action of $G$ on $A_X$ is either trivial or $-1$, therefore there is a well defined extension of the action of $G$ on $\Lambda$: if $g\in G$ acts trivially on $A_X$, then $g(e+3f)=e+3f$, otherwise $g(e+3f)=-e-3f$. We keep calling $T_G(\Lambda):=\Lambda^G$ and $S_G(\Lambda):=T_G(\Lambda)^\perp$. The advantage in this setting is that if $G$ is of prime order $p$, then $T_G(\Lambda)$ and $S_G(\Lambda)$ are $p$-elementary lattices. This is by no means true for $S_G(X)$ and $T_G(X)$. To prove this last statement we use the following lemma:

\begin{lem}\label{lem:G_tors}
Let $R$ be a lattice, and let $G\subset O(R)$. Then the following hold:
\begin{itemize}
\item $T_G(R)$ contains $\sum_{g\in G}gv$ for all $v\in R$.
\item $S_G(R)$ contains $v-gv$ for all $v\in R$ and all $g\in G$.
\item $R/(T_G(R)\oplus S_G(R))$ is of $|G|$-torsion.
\end{itemize}
\end{lem}
\begin{proof}
It is obvious that $\sum_{g\in G}gv$ is $G$-invariant for all $v\in R$. For $w\in T_G(R)$ we have $(w,v)=(gw,gv)=(w,gv)$ for all $v\in R$ and all $g\in G$. Therefore $v-gv$ is orthogonal to all $G$-invariant vectors, hence it lies in $S_G(R)$.
Let $t\in R$, we can write $|G|t=\sum_{g\in G} g(t) + \sum_{g\in G}(t-g(t))$, where the first term lies in $T_G(R)$ and the second in $S_G(R)$.
\end{proof}

Since $T_G(\Lambda)$ and $S_G(\Lambda)$ are $p$-elementary lattices, we will use the classification results of Nikulin to find them. The first theorem deals with the case $p=2$:

\begin{thm}[\cite{Nik83}]\label{thm:2el}
An even hyperbolic $2$-elementary lattice of rank $r$ is uniquely determined by the invariants $(r,a,\delta)$ and exists if and only if the following conditions are satisfied:
$$\left\{
\begin{array}{l l}
a\leq r\\
r\equiv a\mod{2}\\
\text{if }\delta=0\text{, then }r\equiv 2\mod{4}\\
\text{if }a=0\text{, then }\delta=0\\
\text{if }a\leq 1\text{, then }r\equiv 2\pm a\mod{8}\\
\text{if }a=2\text{ and }r\equiv 6\mod{8}\text{, then }\delta=0\\
\text{if }\delta=0\text{ and }a=r\text{, then }r\equiv 2\mod{8}
\end{array}
\right.$$
\end{thm}

And the second theorem deals with the case $p\neq 2$:

\begin{thm}[{\cite[Section 1]{RS81}}]\label{thm:pel}
An even hyperbolic $p$-elementary lattice of rank $r$ with $p\neq 2$ with invariants $(r,a)$ exists if and only if the following conditions are satisfied:
$$\left\{
\begin{array}{l l}
a\leq r\\
r\equiv 0\mod{2}\\
\text{if }a\equiv 0\mod{2}\text{, then }r\equiv 2\mod{4}\\
\text{if }a\equiv 1\mod{2}\text{, then }p\equiv (-1)^{r/2-1}\mod{4}\\
\text{if }r\not\equiv 2\mod{8}\text{, then }r>a>0
\end{array}
\right.$$
Such a lattice is uniquely determined by the invariants $(r,a)$ if $r\geq 3$.
\end{thm}

Since the classification theorem of $p$-elementary lattices with $p\neq 2$ above deals only with hyperbolic lattices, we will need sometimes to split a lattice to study it. This is done by the following theorem:

\begin{thm}[{\cite[Corollary~1.13.5]{Nik79}}]\label{thm:split}
Let $L$ be an even lattice of rank $r$. If $L$ is indefinite, and $r\geq 3+1(A_L)$, then $L\simeq U\oplus L_0$ for some lattice $L_0$.
\end{thm}

The following Lemma gives some restrictions for a lattice to have an action of prime order:

\begin{lem}\label{lem:aleqm}
Let $L$ be a lattice and $G\subset O(L)$ a subgroup generated by $\varphi$ of prime order $p$. Then
\[m:=\frac{\rk\left(S_G(L)\right)}{p-1}\]
is an integer and
\[\frac{L}{T_G(L)\oplus S_G(L)}\cong\left(\Z/p\Z\right)^a.\]
Moreover, there are natural embeddings of $\frac{L}{T_G(L)\oplus S_G(L)}$ into the discriminant groups $A_{T_G(L)}$ and $A_{S_G(L)}$, and $a\leq m$.
\end{lem}

\begin{proof}
By \cite[Theorem 74.3]{CR88} we have the following isomorphism of $\Z[G]$-modules:
\begin{equation*}
L\cong\bigoplus_{i=1}^s(A_i,a_i)\oplus\bigoplus_{j=s+1}^{s+t}\oplus\Z^{\oplus u}
\end{equation*}
where the $A_i$ are ideals of $\Z[\xi_p]$ with $\xi_p$ a primitive $p$-th root of the unity, $(A_i,a_i)\cong A_i\oplus\Z$ as $\Z$-modules with $a_i\in A_i\setminus(\xi_p -1)A_i$, and the action of $\varphi$ is defined by:
\begin{itemize}
\item trivial on $\Z^{\oplus u}$,
\item by multiplication by $\xi_p$ on $A_i$,
\item and by $\varphi.(x,k)=(\xi_p x+a_i k,k)$ on $(A_i,a_i)$.
\end{itemize}
Since this decomposition respects the action of $\varphi$, we can look for invariant and coinvariant lattices on each term:
\begin{itemize}
\item The action is trivial on $\Z^{\oplus u}$ so this term goes in the invariant part.
\item On $A_i$, $\varphi$ acts by $\xi_p$. But $S_G(L)=\Ker\left(\sum_{j=0}^{p-1}\varphi^j\right)$ and this morphism acts on $A_i$ by $\sum_{j=0}^{p-1}\xi_p^j=0$, hence all $A_i$ are coinvariant.
\item On $(A_i,a_i)$, we compute $\sum_{j=0}^{p-1}\varphi^j.(x,k)=(*,pk)$ so this can be zero only if $k=0$. Moreover, the action restricts to $A_i$ by multiplication by $\xi_p$, hence by the same argument as above, $A_i$ is the coinvariant part of $(A_i,a_i)$. Then we look for invariant elements of $(A_i,a_i)$, that is elements such that:
\begin{align*}
&\varphi.(x,k)=(\xi_p x+a_i k,k)=(x,k)\\
\Longleftrightarrow&(\xi_p -1)x=-ka_i \numberthis \label{eqcont}
\end{align*}
so $\xi_p -1$ divides $ka_i$. But $p=(\xi_p -1)\sum_{j=1}^{p-1}j\xi_p^j$, hence $y_i:=\left(-\frac{p}{\xi_p -1}a_i,p\right)$ is fixed by the action and $\Z y_i$ is in the invariant part $T_G(A_i,a_i)$ of $(A_i,a_i)$. Then we get:
\begin{equation*}
\frac{(A_i,a_i)}{S_G(A_i,a_i)\oplus T_G(A_i,a_i)}\cong\left.\left(\frac{(A_i,a_i)}{A_i\oplus\Z y_i}\right)\right/H\cong\left.\left(\Z/p\Z\right)\right/H
\end{equation*}
where $H$ is isomorphic to $\{0\}$ or $\Z/p\Z$. But if it was $\Z/p\Z$, then we would have $(A_i,a_i)=S_G(A_i,a_i)\oplus T_G(A_i,a_i)$ and the element $(0,1)$ of $(A_i,a_i)$ would decompose into $(-x,0)+(x,1)$, with $(-x,0)\in S_G(A_i,a_i)$ and $(x,1)\in T_G(A_i,a_i)$. By the Equation~\ref{eqcont}, this would mean that $(\xi_p -1)x=-a_i$ and this is impossible because $a_i\notin(\xi_p -1)A_i$. So $H=\{0\}$, $\frac{(A_i,a_i)}{S_G(A_i,a_i)\oplus T_G(A_i,a_i)}\cong\Z/p\Z$ and $T_G(A_i,a_i)=\Z y_i$.
\end{itemize}
Coming back to $L$, what we did before implies that $T_G(L)=\bigoplus_{i=1}^s\Z y_i\oplus\Z^{\oplus u}$ and $S_G(L)=\bigoplus_{i=1}^{s+t}A_i$, hence:
\begin{equation*}
\frac{L}{S_G(L)\oplus T_G(L)}\cong\bigoplus_{i=1}^s\frac{(A_i,a_i)}{A_i\oplus\Z y_i}\cong\left(\Z/p\Z\right)^{\oplus s}.
\end{equation*}
We put $a:=s$ and, since $T_G(L)$ and $S_G(L)$ are primitive in $L$ and $S_G(L)=T_G(L)^\perp$, by \cite[\S 1, 5]{Nik79} we get a natural embedding $\frac{L}{S_G(L)\oplus T_G(L)}\subset A_{T_G(L)}\oplus A_{S_G(L)}$ and the projections on $A_{S_G(L)}$ and $A_{T_G(L)}$ give the embeddings that we wanted.\\
Finally, for any $i$ we have $\rk(A_i)=\dim_\Q(A_i\otimes\Q)=\dim_\Q(\Q[\xi_p])=p-1$, hence $m:=\frac{\rk\left(S_G(L)\right)}{p-1}=s+t$ is an integer and $a=s\leq s+t=m$.
\end{proof}

The following theorems will also be useful along the classification:

\begin{thm}[{\cite[Theorem~1.13.3]{Nik79}}]\label{thm:uniclat}
Let $L$ be an even lattice with signature $(r_+,r_-)$ and discriminant form $q_L$. If $L$ is indefinite and $l(A_L)\leq\rk(L)-2$, then $L$ is the only lattice up to isometry with the invariants $(r_+,r_-,q_L)$.
\end{thm}

\begin{prop}[{\cite[Section 4]{BCMS14}}]\label{prop:BCMS14p}
Let $L$ be a lattice with a non-trivial action of order $p$, with rank $p-1$ and discriminant $d_L$, then $\frac{d_L}{p^{p-2}}$ is a square in $\Q$.
\end{prop}

\subsection{Automorphisms}

\subsubsection{Induced automorphisms}
We recall briefly the notion of \em induced automorphisms \em on Kummer fourfolds. For a more detailed and systematic introduction we refer to \cite{MW14}.

Let $A$ be an abelian surface, $\varphi$ a group automorphism of $A$ and $v=(r,l,s)\in H^0(A)\oplus \NS(A)\oplus H^4(A)$ a so-called \em Mukai vector \em satisfying $r\geq 0,$ $l$ effective for $r=0$ and $v^2:=l^2-2rs=6$ as before. Suppose $\varphi^*l=l$ and assume that there exists a $\varphi$-invariant $v$-generic stability condition. Then the Albanese fibre of the associated moduli space of stable objects $M(v)$ is a fourfold of Kummer type by \Ref{prop}{mod_spac} and the pullback $\varphi^*$ induces an automorphism on $M(v)$, called the \em induced automorphism \em of $\varphi$ and $v$.

The main result about induced automorphisms on fourfolds of Kummer type in \cite{MW14} is the following lattice-theoretic characterisation:

\begin{thm}[{\cite[Thm.\ 3.5]{MW14}}]
Let $X$ be a fourfold of Kummer type and $\sigma$ an automorphism of $X$. If the induced action of $\sigma$ on $A_X$ is trivial and the isometry of $\Lambda$ (which is induced via the $\sigma$-equivariant embedding $H^2(X,\Z)\hookrightarrow \Lambda$) fixes a copy of $U$, then $X$ is the Albanese fibre of a moduli space of stable objects and $\sigma$ is an induced automorphism as introduced above.
\end{thm}

Note that it is a straightforward calculation that every induced automorphism satisfies the conditions in the theorem above.

If we consider the special mukai vector $v=(1,0,3)$, one of the associated moduli spaces is nothing but the generalised Kummer fourfold, i.e.\ the Albanese fibre of the Hilbert scheme $A^{[3]}$. The induced automorphisms in this case are called \em natural automorphisms. \em They have been studied in detail in \cite{Kevin}.

\subsubsection{Automorphisms from lattice isometries}
To construct automorphisms directly is, in general, very difficult. Using the Hodge-theoretic Torelli theorem (\Ref{thm}{torelli}), we can show that certain lattice isometries of our reference lattice $L:=U^{\oplus3}\oplus\lat{-6}$ actually come from automorphisms of Kummer type fourfolds. Let us illustrate this by fixing an isometry of $L$ of order $p$, which generates a cyclic subgroup $G\subset Mon^2(L)$ (the group $Mon^2(L)$ is obtained from $Mon^2(X)$ using an arbitrary marking). Assume that the invariant lattice $T_G(L)\subset L$ has signature $(1,*)$. Now, let $X$ be a fourfold of Kummer type together with a marking such that the invariant lattice of the induced isometry group $G_X$ on $H^2(X,\Z)$ is contained in $\NS(X)$. The group $G_X$ therefore consists of parallel transport operators and we obtain a group of birational self-maps of $X$. If furthermore $G_X$ preserves all wall divisors in $NS(X)$, then we actually constructed automorphisms of $X$, which, in this case, can easily seen to be non-symplectic. Thus, in the following sections we seek to classify all such isometries of $L$.

\renewcommand{\arraystretch}{1.5}
\section{Non-symplectic automorphisms I: Classification for $\Lambda$}\label{sec:nonsympl}
Suppose $G$ is cyclic and generated by a non-symplectic element of maximal order $m$.

\begin{cor}[{[Oguiso-Schr\"{o}er, \cite{OS11}]}]
We have $m\leq 18$ and $\phi(m)\leq 6$.
\end{cor}

Moreover, in this case $T_G(X)$ has signature $(1,\rk T_G(X)-1)$ (because it contains an invariant ample class) and $S_G(X)$ has signature $(2,\rk S_G(X)-2)$.
Let now $G\cong \Z/p\Z$ be a cyclic group of non-symplectic automorphisms on a fourfold $X$ of Kummer type. By the considerations at the beginning of \Ref{subsec}{lat} we obtain a $G$-equivariant embedding of $H^2(X,\Z)$ into $\Lambda$. Conversely, $H^2(X,\Z)$ can be considered as the orthogonal complement of a square $6$ vector $v\in\Lambda$, which will be referred to as the \em Mukai vector \em for $X$. If the action of $G$ on $A_X$ is trivial, we have
\[ S_G(X)\cong S_G(\Lambda) \quad \text{and}\quad T_G(X)\cong v^\perp\subset T_G(\Lambda).\]

Thus $S_G(\Lambda)$ and $T_G(\Lambda)$ have signature $(2,*)$.

If $G$ acts as $-1$ on $A_X$ -- this can only happen for $p=2$ -- we have
\[ T_G(X)\cong T_G(\Lambda) \quad \text{and} \quad S_G(X)\cong v^\perp\subset S_G(\Lambda).\]
Thus in this case $S_G(\Lambda)$ has signature $(3,*)$ and $T_G(\Lambda)$ has signature $(2,*)$.

It follows from \Ref{lem}{G_tors} that all lattices $S_G(\Lambda)$ and $T_G(\Lambda)$ are $p$-elementary, thus their descriminant group is isomorphic to $(\Z/p\Z)^a$ for some integer $a\geq0$.

Now, in order to classify non-symplectic automorphisms of fourfolds of Kummer type of prime order $p$ in the next section, we first study $p$-elementary sublattices of $\Lambda\cong U^{\oplus 4}$ that might occur as (co-)invariant lattices of isometries of $\Lambda$ of order $p$. Note that this list, in principal, can be used to classify non-symplectic automorphisms of manifolds of Kummer type of any dimension ($\geq 4$). In the case $p=2$ we add a column '$\delta$' to indicate whether the quadratic form of the discriminant group of the lattices at hand is integer valued ($\delta=0$) or not ($\delta=1$).

\subsection{{\boldmath$p=2$} and {\boldmath$\textrm{sign}(S_G(\Lambda))=(2,*)$}}
\begin{prop}\label{prop:lattice_lambda_2a}
The following is a complete list of coinvariant lattices ($S_G(\Lambda)$) and invariant lattices ($T_G(\Lambda)$) of even rank and signature $(2,*)$ of order two isometries of $\Lambda$.
\end{prop}

We call an isometry \em induced \em if $T_G(\Lambda)$ contains a copy of $U$.

\begin{tabular}{|c|c|c|c|c|c|c|}\hline
No. & $S_G(\Lambda)$ & $T_G(\Lambda)$  & $\rk(T_G(\Lambda))$ & $a$ & $\delta$ & is induced\\\hline\hline
$1$ & $U^{\oplus2}\oplus \lat{-2}^{\oplus2}$& $\lat{2}^{\oplus2} $& $2$& $2$ &$1$ & no\\\hline
\hline
$2$ & $U^{\oplus2}$ & $U^{\oplus2}$ & $4$ & $0$ & $0$ & yes\\\hline
$3$ & $U\oplus U(2)$ & $U\oplus U(2)$& $4$ & $2$ &$0$ & yes\\\hline
$4$ & $U\oplus \lat{2}\oplus \lat{-2}$&$U\oplus \lat{2}\oplus \lat{-2}$& $4$ & $2$ & $1$ &yes\\\hline
$5$ & $U(2)\oplus U(2)$ & $U(2)\oplus U(2)$ & $4$ & $4$ & $0$ & no\\\hline
$6$ & $U(2)\oplus \lat{2}\oplus \lat{-2}$ & $U(2)\oplus \lat{2}\oplus \lat{-2}$ & $4$ & $4$ &$1$ & no\\\hline
\hline
$7$ & $\lat{2}^{\oplus2} $ &$ U^{\oplus2}\oplus \lat{-2}^{\oplus2}$ & $6$ & $2$ &$1$ & yes\\\hline
\end{tabular}

\begin{proof}
When $\rk(T_G(\Lambda))=2$, $T_G(\Lambda)$ is positive definite, then we can find those lattices in the list \cite[Table~15.1, p.360]{CS99}. When $\rk(T_G(\Lambda))=4$ or $6$, we apply \Ref{thm}{2el} to find the lattices. Since the list above is symmetric in $S$ and $T$ we proceed in the same way for $S_G(\Lambda)$.
\end{proof}

\subsection{{\boldmath$p=2$} and {\boldmath$\textrm{sign}(S_G(\Lambda))=(3,*)$}}\label{ssec:2var} Now we consider the case where $S_G(\Lambda)$ has signature $(3,*)$ and $T_G(\Lambda)$ signature $(1,*)$.

\begin{tabular}{|c|c|c|c|c|c|}\hline
No. & $S_G(\Lambda)$ & $T_G(\Lambda)$ & $\rk(T_G(\Lambda))$ & $a$ & $\delta$ \\\hline\hline
$1$ & $U^{\oplus3}$ & $U$ & $2$ & $0$ & $0$ \\\hline
$2$ & $U^{\oplus2}\oplus U(2)$ & $U(2)$ & $2$ & $2$ & $0$ \\\hline
$3$ & $U^{\oplus2}\oplus \lat{2}\oplus \lat{-2}$ & $\lat{2}\oplus \lat{-2}$ & $2$ & $2$ & $1$ \\\hline\hline
$4$ & $U\oplus \lat{2}^{\oplus2}$ & $U\oplus \lat{-2}^{\oplus2}$ & $4$ & $2$ & $1$\\\hline
$5$ & $\lat{2}^{\oplus3}\oplus \lat{-2}$ & $\lat{2}\oplus \lat{-2}^{\oplus3}$  & $4$ & $4$ & $1$ \\\hline
\end{tabular}

\begin{proof}
This is a direct application of \Ref{thm}{2el}. 
\end{proof}

\subsection{{\boldmath$p=3$}}
The following is a complete list of coinvariant lattices ($S_G(\Lambda)$) and invariant lattices ($T_G(\Lambda)$) of even rank and signature $(2,*)$ of order three isometries of $\Lambda$. We call an isometry \em induced \em if $T_G(\Lambda)$ contains a copy of $U$.

\begin{tabular}{|c|c|c|c|c|c|}\hline
No. & $S_G(\Lambda)$ & $T_G(\Lambda)$ & $\rk(T_G(\Lambda))$ & $a$  & is induced\\\hline\hline
$1$ & $U^{\oplus2}\oplus A_2(-1)$ & $A_2$ & $2$ & $1$ & no\\\hline\hline
$2$ & $U^{\oplus2}$ & $U^{\oplus2}$ & $4$ & $0$ & yes\\\hline
$3$ & $U\oplus U(3)$ & $U\oplus U(3)$ & $4$ & $2$ & yes\\\hline
$4$ & $A_2$ & $U^{\oplus2}\oplus A_2(-1)$ & $6$ & $1$ & yes\\\hline
\end{tabular}

\begin{proof}
When $\rk(T_G(\Lambda))=2$, $T_G(\Lambda)$ is positive definite. We can then find those lattices in the list \cite[Table~15.1, p.360]{CS99} and $S_G(\Lambda)$ is uniquely determined because we can write it as $U\oplus S'$ by \Ref{thm}{split} with $S'$ hyperbolic of rank $4$, which is unique by \Ref{thm}{pel}.\\
When $\rk(T_G(\Lambda))=6$, we can do the same thing as above, exchanging the roles of $S_G(\Lambda)$ and $T_G(\Lambda)$.\\
When $\rk(T_G(\Lambda))=4$, we deal with the cases $a=0$ and $1$ by splitting a copy of $U$ again by \Ref{thm}{split} and looking for the last piece in the list \cite[Table~15.2a, p.362]{CS99}. The only case left is $a=2$ by \Ref{lem}{aleqm}. The isomorphism classes of $3$-elementary discriminant forms of rank $2$ are those of $U(3)$ and $A_2^{\oplus 2}$. But the signature of our lattice must be $(r_+,r_-)=(2,2)$, hence the case of $A_2^{\oplus 2}$ is impossible by \cite[Theorem 1.10.1]{Nik79}, because $r_+-r_-\neq 4$. The only possible discriminant form is the one of $U(3)$. \Ref{thm}{uniclat} applies here, so the lattice is isomorphic to $U\oplus U(3)$.
\end{proof}



\subsection{{\boldmath$p=5$}}
For $p=5$ the rank of $S_G(\Lambda)$ must be four in order to obtain an isometry with determinant $1$. Furthermore we exclude the cases $a=0,2$, using Proposition 2.6 of \cite{MTW15}. We are left with the following example:

\begin{tabular}{|c|c|c|c|c|}\hline
$S_G(\Lambda)$ & $T_G(\Lambda)$ & $\rk(T_G(\Lambda))$ & $a$  & is induced\\\hline\hline
$U\oplus H_5$ & $U\oplus H_5$ & $4$ & $1$ & yes\\\hline
\end{tabular}

\begin{proof}
We begin by looking for $S_G(\Lambda)$. By \Ref{lem}{aleqm}, we always have $m=1$ (so $\rk(S_G(\Lambda))=4$) and $a\leq m$. Moreover, $a$ is odd by \Ref{prop}{BCMS14p}, so $a=1$. We can then split a copy of $U$ by \Ref{thm}{split} and look for the last piece in the list \cite[Table~15.2a, p.362]{CS99}. Since in this case $q_{T_G(\Lambda)}=-q_{S_G(\Lambda)}\simeq q_{S_G(\Lambda)}$, we deduce from \Ref{thm}{uniclat} that $T_G(\Lambda)\simeq S_G(\Lambda)$.
\end{proof}


\subsection{{\boldmath$p=7$}}
Also for $p=7$ there is a unique possibility:

\begin{tabular}{|c|c|c|c|c|}\hline
$S_G(\Lambda)$ & $T_G(\Lambda)$ &  $\rk(T_G(\Lambda))$ & $a$  & is induced\\\hline\hline
$U^{\oplus 2}\oplus K_7(-1)$ & $K_7$ & $2$ & $1$  & no\\\hline
\end{tabular}

\begin{proof}
We begin by looking for $S_G(\Lambda)$. By \Ref{lem}{aleqm}, we always have $m=1$ (so $\rk(S_G(\Lambda))=4$) and $a\leq m$. Moreover, $a$ is odd by \Ref{prop}{BCMS14p}, so $a=1$. We can then split a copy of $U$ by \Ref{thm}{split} and look for the opposite of the last piece (for it to be positive definite) in the list \cite[Table~15.1, p.360]{CS99}. The lattice $T_G(\Lambda)$ is then positive definite with rank $2$ and $a=1$, we find it also in the list \cite[Table~15.1, p.360]{CS99}.
\end{proof}


\begin{oss}
For $p\geq 3$, in principle we could have several isometries of order $p$ which are induced by automorphisms. However, when the covariant lattice is isometric to the covariant lattice of a $K3$ surface with an automorphism of order $p$, we have a unique action on the lattice. Indeed, in \cite[Prop. 9.3]{AST11}, it is proven that the connected components of the moduli space of $K3$ surfaces with an automorphism of order $p$ is given by the isometry classes of the covariant lattice. If we had more than one isometry for each of them, this would be impossible, hence our claim. In particular, all covariant lattices in the above lists for $p=3,5,7$ occur as covariant lattices on $K3$ surfaces.
\end{oss}

\section{Non-symplectic automorphisms II: Classification in dimension four}\label{sec:nonsymp4}
Now, we specialise to the case of fourfolds. Lattice-theoretically this is done by chosing a length $6$ vector $v$ in $\Lambda$ which will serve as a Mukai vector. The cohomology $H^2(X,\Z)$ will then be isometric to $v^\perp$. Some of the invariant or coinvariant lattices appearing in the previous section do not contain any such vector. But some lattices admit two non-conjugate choices. Keeping this in mind, we will give a list of all configurations of lattices occuring as invariant and co-invariant lattices of prime order non-symplectic automorphisms of Kummer 4-type manifolds. Note that for each such configuration there are four families in the moduli space of marked manifolds. Each two of them are geometrically identical, but the marking differs by $-1$. Finally the last family is obtained by composing with the 'dual' isometry. The geometric meaning of this isometry is only understood for induced automorphisms, where we simply construct the 'dual manifold' as the moduli space of objects on the dual abelian variety.

If we choose the vector $v$ to be in the invariant lattice $T_G(\Lambda)$, the induced action on $H^2(X,\Z)$ has a trival action on the discriminant $A_X$. In this situation the discriminant group of $S_G(X)$ is isomorphic to $(\Z/p\Z)^a$ for some integer $a(S_G(X))=a\geq 0$.

If we choose $v$ to belong to the coinvariant lattice $S_G(\Lambda)$, we obtain a non-trivial action on $A_X$. Note that since the action on $A_X$ must be given by $\pm1$ (see \Ref{sec}{prelim}), this can only happen for $p=2$. In this case the discriminant group of the invariant lattice $T_G(X)$ is isomorphic to $(\Z/p\Z)^a$ for some integer $a(T_G(X))=a\geq0.$

In every table we add a column '$\dim$' which gives the dimension of the family of manifolds with the corresponding automorphism.

We indicate by 'ind' all entries corresponding to induced automorphisms and by 'nat' those which are even natural.

The following classification is complete with respect to the possible lattices $S_G(X)$ and $T_G(X)$, however to count the number of connected components of the moduli space of Kummer fourfolds with this prescribed $G$ action, one needs to count embeddings of $v\in T_G(\Lambda)$ (or $S_G(\Lambda)$ in the appropriate cases) up to isometry.

\subsection{{\boldmath$p=2$}, trivial action on {\boldmath$A_X$}}
If the automorphism acts trivially on $A_X$, we have $S_G(X)\cong S_G(\Lambda)$ and $T_G(X)$ is obtained as the orthogonal complement of the square $6$ class $v\in T_G(\Lambda)$. Note that $S_G(X)$ is $2$-elementary.  We denote $a(S_G(X))$ by $a$ and $\div_{T_G(\Lambda)}(v)$ by $d$. 

\vspace{10pt}
\begin{tabular}{|c|c|c|c|c|c|c|c|}\hline
No. & $S_G(X)$ & $T_G(X)$ & $d$ & $a$ & $\delta(S_G(X))$ &  $\dim$ & ind/nat\\\hline\hline
$2$ & $U^{\oplus2}$  & $U\oplus \lat{-6}$& $1$ & $0$ & $0$ & $2$ & nat\\\hline
$3$ & $U\oplus U(2)$ & $U(2)\oplus \lat{-6}$& $1$ & $2$ & $0$ & $2$ & nat\\\hline
$4.1$ & $U\oplus \lat{2}\oplus \lat{-2}$& $\lat{2}\oplus \lat{-2}\oplus \lat{-6}$& $1$ & $2$ & $1$ & $2$ & nat\\\hline
$4.2$ & $U\oplus \lat{2}\oplus \lat{-2}$& $\lat{2}\oplus A_2(-1)$ & $2$ & $2$ & $1$  & $2$ & ind\\\hline
$6.1$ & $U(2)\oplus \lat{2}\oplus \lat{-2}$ & $U(2)\oplus \lat{-6}$ & $2$ & $4$ & $1$ & $2$ & no\\\hline
\hline
$7.1$ & $\lat{2}^{\oplus2} $ & $U\oplus \lat{-2}^{\oplus2} \oplus\lat{-6}$ & $1$ & $1$  & $1$ & $0$ & nat\\\hline
$7.2$ & $\lat{2}^{\oplus2}$ & $U\oplus \lat{-2}\oplus A_2(-1)$ & $2$ & $2$ & $1$ & $0$ & ind\\\hline

\end{tabular}


\subsection{{\boldmath$p=2$}, non-trivial action on {\boldmath$A_X$}}
If the automorphism acts as $-1$ on $A_X$ we have $T_G(X)\cong T_G(\Lambda)$ and $S_G(X)$ is obtained as the orthogonal complement of a square $6$ element $v\in S_G(\Lambda)$. Note that $T_G(X)$ is $2$-elementary. We denote $a(T_G(X))$ by $a$ and $\div_{S_G(\Lambda)}(v)$ by $d$.

Note that none of these examples can be realised by an induced automorphism. The column 'MS' indicates whether the manifolds in this family or moduli spaces of stable objects or not. We denote by 'LFwS' those families of moduli spaces which are (more strongly) relative compactified Jacobians admitting a lagrangian fibration with section.

\vspace{10pt}
\begin{tabular}{|c|c|c|c|c|c|c|c|}\hline
No. & $S_G(X)$ & $T_G(X)$ & $d$ & $a$ & $\delta(T_G(X))$ & $\dim$ & MS\\\hline\hline
$1$ & $U^{\oplus2}\oplus \lat{-6}$ & $U$ & $1$ & $0$ & $0$ & $3$ & yes\\\hline
$2$ & $U\oplus U(2)\oplus\lat{-6}$ & $U(2)$ & $1$ & $2$ & $0$ & $3$ & no\\\hline
$3.1$ & $U\oplus\lat{2}\oplus \lat{-2}\oplus \lat{-6}$& $\lat{2}\oplus \lat{-2}$ & $1$ & $2$ & $1$  & $3$ & no\\\hline
$3.2$ & $U^{\oplus2}\oplus\lat{-6}$ & $\lat{2}\oplus \lat{-2}$ & $2$ & $2$ & $1$  & $3$ & LFwS\\\hline
\hline
$4$ & $\lat{2}^{\oplus2}\oplus \lat{-6}$ & $U\oplus \lat{-2}^{\oplus2}$ & $1$ & $2$ & $1$  & $1$ & yes\\\hline
$5.1$ & $\lat{2}^{\oplus2}\oplus\lat{-6}$ & $\lat{2}\oplus \lat{-2}^{\oplus3}$ & $2$ & $4$ & $1$  & $1$ & LFwS\\\hline
$5.2$ & $A_2(2)\oplus\lat{-2}$ & $\lat{2}\oplus \lat{-2}^{\oplus3}$ & $2$ & $4$ & $1$  & $1$ & yes\\\hline
\end{tabular}


\subsection{{\boldmath$p=3$}}

For $p=3$ the action on $A_X$ is trivial, hence $S_G(\Lambda)=S_G(X)$ and $T_G(\Lambda)$ have signature $(2,*)$ and the latter is obtained as the orthogonal of a square $6$ class $v\in T_G(\Lambda)$.

\vspace{10pt}
\begin{tabular}{|c|c|c|c|c|c|c|}\hline
No. & $S_G(X)$ & $T_G(X)$ & $\div_{T_G(\Lambda)}(v)$ & $a(S_G(X))$ & $\dim$ & ind/nat\\\hline\hline
$1$ & $U^2\oplus A_2(-1)$  & $\lat{2}$ & $2$& $1$ & $2$ & no\\\hline
\hline
$2$ & $U^{\oplus2}$  & $U\oplus \lat{-6}$ &$1$ & $0$ & $1$ & nat\\\hline
$3.1$ & $U(3)\oplus U$  & $U(3)\oplus \lat{-6}$ & $1$ & $2$ & $1$ & nat\\\hline
$3.2$ & $U(3)\oplus U$  & $U\oplus \lat{-6}$ & $3$ & $2$ & $1$ & ind\\\hline
\hline
$4.1$ & $A_2$  & $U\oplus A_2(-1)\oplus \lat{-6}$ &$1$& $1$ & $0$ & nat\\\hline
$4.2$ & $A_2$ & $U^{\oplus2}\oplus \lat{-2}$&$2$ & $1$ & $0$ & ind\\\hline
\end{tabular}



\subsection{{\boldmath$p=5$}}
For $p=5$  Moreover, the determinant must be $1$ on $H^2(X)$ to ensure that it is a monodromy operator, thus $S_G$ has even rank. There is only one family which is induced and where the divisor of $v$ is one:

\vspace{10pt}
\begin{tabular}{|c|c|c|c|c|c|}\hline
$S_G(X)$ & $T_G(X)$ &$\div_{T_G(\Lambda)}(v)$ & $a(S_G(X))$ & $\dim$ & ind/nat\\\hline\hline
$U\oplus H_5$ & $H_5\oplus \lat{-6}$ & $1$ & $1$ & $0$ & nat\\\hline
\end{tabular}


\section{Examples of non-symplectic automorphisms}
In this section we collect all examples of non-symplectic automorphisms on fourfolds of Kummer type.

\subsection{{\boldmath$p=2$}, trivial action on \boldmath{$A_X$}}
\begin{ex}[{[Invariant lattices $U\oplus\lat{-6}$, $U(2)\oplus\lat{-6}$, $\lat{2}\oplus\lat{-2}\oplus\lat{-6}$ and $U\oplus\lat{-2}^{\oplus2}\oplus\lat{-6}$]}]
These are the natural automorphisms corresponding to the automorphisms of abelian surfaces in Section 4.1 of \cite{MTW15}.
\end{ex}

\begin{ex}[{[Invariant lattices $\lat{2}\oplus A_2(-1)$ and $U\oplus\lat{-2}\oplus A_2(-1)$]}]
These are both induced automorphism which are not natural. The underlying abelian surfaces have invariant lattices (which generically is isomorphic to the Picard lattice) isometric to $\lat{2}\oplus\lat{-2}$ and $U\oplus\lat{-2}^{\oplus2}$ respectively. The examples can be constructed as follows: Choose an ample class $H$ on the abelian surface of square $6$ which has divisibility $2$ inside the invariant lattice. The Albanese fibre of the moduli space of sheaves with Mukai vector $v=(0,H,0)$ is a fourfold of Kummer type and carries the induced automorphism from the surface. Note that the moduli spaces occurring here are relative compactified Jacobians over the linear system $H$ of genus $4$.
\end{ex}

\subsection{{\boldmath$p=2$}, non-trivial action on \boldmath{$A_X$}}
\begin{ex}[{[Invariant lattices $\lat{2}\oplus\lat{-2}$ and $\lat{2}\oplus \lat{-2}^{\oplus3}$]}]
Let $A$ be an abelian surface admitting a polarisation of type $(1,3)$, i.e.\ there is an ample class $H$ of square $6$ yielding a $6:1$-cover of $\PP^2$. If we choose the Mukai vector $v=(0,H,-3)$ the corresponding moduli space $M(v)$ is the relative compactified Jacobian of degree $0$ over the genus $4$ linear system $|H|=(\PP^2)^\vee$. The involution which is given by $-1$ on the smooth fibres restricts to a lagrangian fibration with section on the Albanese fibre $X$ and yields thus an example of a non-symplectic involution on fourfolds of Kummer type. If $A$ has Picard rank one, a direct calculation shows that 
\[NS(X)\cong v^\perp\cap\, \widetilde{H}^{1,1}(A)\cong \lat{2}\oplus\lat{-2}\] and $S_G(X)\cong U^{\oplus2}\oplus\lat{-6}$ as coinvariant lattice.

If we degenerate $A$ such that $\NS(A)\cong \lat{6}\oplus\lat{-2}^{\oplus2}$, we generically obtain for the Albanese fibre $NS(X)\cong \lat{2}\oplus\lat{-2}^{\oplus3}$. Thus $X$ still possesses a lagrangian fibration with section which is preserved by the involution. In fact, the involution acts fibrewise. 
\end{ex}

\begin{ex}[{[Invariant lattices $U$ and $U\oplus\lat{-2}^{\oplus2}$]}]
We consider the same surfaces as in the preceding example but this time we choose the Mukai vector $v=(0,H,0)$. Thus our manifolds are Albanese fibres of relative compactified Jacobians of degree $3$ this time, which again is fibred over $|H|$. The automorphism -- let us call it $\sigma$ -- is again given by $-1$ on the smooth fibres. Indeed, it respects the lagrangian fibration, inducing an involution of the base $|H|= (\PP^2)^\vee$. A non-trivial involution on $\PP^2$ must fix a line and a point, on the other hand, this involution must respect the stratification of $|H|$ by the type of singularity of the corresponding curves. In particular, it preserves the dual curve of the branching $A\rightarrow \PP^2$. Both curves are of degree $18$. Thus we conclude that $\sigma$ acts trivially on $|H|$ and we have a fibrewise action which can only be given by $-1$. (Note that translations by two-torsion points yield symplectic automorphism of the manifold.) 


As before we obtain the second family by the degeneration of our abelian surface $A$ such that $NS(A)\cong \lat{6}\oplus \lat{-2}^{\oplus2}$.

\end{ex}

\subsection{{\boldmath$p=3$}}
\begin{ex}[{[Invariant lattices $U\oplus\lat{-6}$ ($\div=1$), $U(3)\oplus\lat{-6}$, $U\oplus A_2(-1)\oplus\lat{-6}$]}]
Again, these are the natural automorphisms corresponding to \cite[Sect.\ 4.2]{MTW15}
\end{ex}

\begin{ex}[{[Invariant lattices $U\oplus\lat{-6}$ ($\div=6$) and $U^{\oplus2}\oplus\lat{-2}$]}]
Again, these are induced but not natural examples. The corresponding invariant lattices for the abelian surfaces are $U(3)$ and $U\oplus A_2(-1)$ (respectively). Again we can choose ample classes of square $6$ with divisibility $6$ and $2$ (resp.).
\end{ex}

\subsection{{\boldmath$p=5$}}
The only example occuring is natural, corresponding to Example 3.12 in \cite{MTW15}.

\subsection{Natural automorphisms}
In the particular case of natural automorphisms, the study can be made much more precise. In fact, one can compute the fixed loci of those automorphisms. This is done in \cite{Kevin} for all natural automorphisms of Kummer fourfolds. Let us sketch the computation of the fixed locus in one example.

\begin{ex}[{[$p=3$ and invariant lattice $U\oplus A_2(-1)\oplus\lat{-6}$]}]
This automorphism comes from Example 3.11 of \cite{MTW15}. Let $E_6=\C/\lat{1,\xi_6}$ where $\xi_6=e^{\frac{i\pi}{3}}$,  and consider the abelian surface $A=E_6\times E_6$ with the automorphism $\varphi=(\xi_3,\xi_3)$. We want to compute the fixed locus of the natural automorphism $\varphi^{[3]}$ on the generalised Kummer $K^3(A)$, coming from $\varphi$ on $A$. The fixed locus will consist of points supported by $\{z,z',z''\}$ in the Hilbert scheme of $3$ points $A^{[3]}$. When $z$, $z'$ and $z''$ are not distinct, the fixed locus contains fat points with ideals of colength $2$ or $3$ fixed by the action of $\varphi$.

There are two possibilities for the support of fixed points of $\varphi^{[3]}$. The first one is of the form $\{z,\varphi(z),\varphi^2(z)\}$, where $z$ is any point in $A$. When $z$ varies in $A$, this gives a connected component of the fixed locus. In particular, in the limit $z\rightarrow z_0,$ where $z_0$ is one of the $9$ fixed points of $\varphi$ in $A$, the support becomes $\{z_0,z_0,z_0\}$. We end up with a copy of $\PP^1$ of points in $K^3(A)$. In fact, this connected component is biholomorphic to a minimal resolution of the $9$ singularities of $A/\lat{\varphi}$, which is a smooth surface of Euler characteristic $15$.

The second possibility for the support is of the form $\{z,z',z''\},$ where $z$, $z'$ and $z''$ are fixed points of $\varphi$ and $z+z'+z''=0$. This gives $21$ additional fixed points of $\varphi^{[3]}$. Hence, the fixed locus of $\varphi^{[3]}$ consists of a biholomorphic copy of a minimal resolution of the singularities of $A/\lat{\varphi}$ and $21$ isolated points.
\end{ex}

\section{Symplectic automorphisms}\label{sec:sympl}
Let now $G$ be a group of symplectic automorphisms. Then we have the following:
\begin{lem}\label{lem:algaction}
Let $X$ be a Kummer type fourfold and let $G\subset Aut(X)$ be a finite symplectic group. Then the following assertions are true:
\begin{enumerate}
\item $S_G(X)$ is nondegenerate and negative definite.
\item $T(X)\subset T_G(X)$ and $S_G(X)\subset S(X)$.
\item $S_G(X)$ contains no wall divisors.
\item The action of $G$ on $A_{S_G(X)}$ is trivial.
\end{enumerate}
\begin{proof}
The proof of the first three items is taken from \cite[Lemma 3.5]{Mon14}, we sketch it here for the reader's convenience. The invariant lattice $T_G(X)$ contains $T(X)$ because $G$ is symplectic and, after tensoring with $\mathbb{R}$, it contains an invariant \kahl class because $G$ is finite. Therefore its orthogonal $S_G(X)$ is negative definite. Since $T_G(X)\otimes \mathbb{R}$ contains a \kahl class, its orthogonal can not contain wall divisors.

The proof of the last item for the \kntiposp case works well also in the generalised Kummer case. Indeed, let us suppose that $G$ does not act trivially on $A_{S_G(X)}$. Then, the proof of \cite[Lemma 3.4]{Mon14} can be applied to our case to find a wall divisor inside $S_G(X)$, which gives the desired contradiction. 
\end{proof}
\end{lem}
In particular, we always have a group of symplectic automorphisms acting trivially on the second cohomology: the kernel of the map $\nu\,:\,Aut(X)\rightarrow O(H^2(X,\mathbb{Z}))$ is a deformation invariant and was determined by Oguiso in \cite{Ogi12}. It is isomorphic to $(\mathbb{Z}/(3))^4\rtimes \mathbb{Z}/(2)$. From now on, when we speak of a group of symplectic automorphisms we actually speak of its image via $\nu$.

\begin{prop}
Let $G\subset Aut(X)$ be a finite group of symplectic automorphisms, then $G\subset O(E_8)$ and $S_G(X)\cong S_G(E_8)(-1)$.
\begin{proof}
As proven in \Ref{lem}{algaction}, $G$ acts trivially on the discriminant group of $H^2(X,\mathbb{Z})$, therefore its action can be extended to the lattice $\Lambda\supset H^2(X,\mathbb{Z})$ trivially outside of $H^2$. Let $r$ be the rank of $S_G(X)$. Its orthogonal $T_G(\Lambda)$ has rank $8-r$, signature $(4,4-r)$ and discriminant form which is the opposite of that of $S_G(X)$. This means that $S_G(X)$ can be embedded in a unimodular negative definite lattice of rank $8$, which means $S_G(X)\subset E_8(-1)$. As $G$ acts trivially on $A_{S_G(X)}$, its action can be extended to an action on $E_8(-1)$ which is trivial on $S_G(X)^\perp$, which is our claim.
\end{proof}
\end{prop}

Moreover we have the following converse:
\begin{prop}\label{prop:e8tosympl}
Let $G\subset O(E_8)$ be a group whose elements have determinant $1$ and suppose there exists a primitive embedding of $S_G(E_8(-1))$ in $U^3\oplus (-6)$. Let $X$ be a Kummer type fourfold such that $Pic(X)\cong S_G(E_8(-1))$ under the above embedding and $Pic(X)$ contains no wall divisors. Then $G$ is induced by symplectic automorphisms of $X$. 
\end{prop}
\begin{proof}
Let $X$ be as above. As it has no wall divisors, its \kahl cone coincides with the positive cone. The action of $G$ on $A_{S_G(X)}$ is trivial, therefore we can extend $G$ to a group of Hodge isometries which are trivial on $T(X)$. Each of these isometries preserves the positive cone, whence also the \kahl cone and therefore \Ref{thm}{torelli} implies that they are induced by automorphisms of the manifold $X$. As their action on $T(X)$ is trivial, they are symplectic automorphisms.
\end{proof}

As in the previous sections, we first analyze subgroups of $O(E_8)$ with elements of determinant one such that their covariant lattice can be embedded in $U^4$, and afterwards we specialize to $U^3\oplus (-6)$ and check the additional conditions of \Ref{prop}{e8tosympl}. To this end, we use the following result.
\begin{thm}\cite[Thm. 3.6]{HM15}
Let $G$ be a subgroup of $O(E_8)$ which is the stabilizer of some sublattice of $E_8$. Then $G$ is the coxeter group of a Dynkin sublattice of $E_8$ and $S_G(E_8)$ is the above said Dynkin lattice. 
\end{thm}
In the following table, we list all these lattices up to rank 4, together with their group of determinant one isometries which act trivial on the discriminant group. Elements with no such isometries are omitted. To denote the groups, we use the following notation: $n^m$ denotes the cartesian product of $m$ cyclic groups of order $n$, $G.H$ denotes an extension of $G$ by $H$ and $\mathfrak{S}_n,\mathfrak{A}_n$ denote symmetric and alternating groups.

\vspace{5pt}
\begin{tabular}{|c|c|c|}\hline
Rank $S_G(X)$ & $G$ & $S_G(E_8)$\\\hline
2 & $2$ & $A_1^2$\\\hline
2 & $3$ & $A_2$\\\hline
3 & $2^2$ & $A_1^3$\\\hline
3 & $\mathfrak{S}_3$ & $A_1 \oplus A_2$\\\hline
3 & $\mathfrak{A}_4$ & $A_3$\\\hline
4 & $2^3$ & $A_1^4$ \\\hline
4 & $2.\mathfrak{S}_3$ & $A_1^2 \oplus A_2$\\\hline
4 & $\mathfrak{S}_4$ & $A_1 \oplus A_3$\\\hline
4 & $2.3^2$ & $A_2^2$\\\hline
4 & $\mathfrak{A}_5$ & $A_4$\\\hline
4 & $2^3.\mathfrak{A}_4$ & $D_4$\\\hline
\end{tabular}
\vspace{5pt}

In order to obtain the classification of symplectic automorphisms, we can now proceed by checking which of these lattices admits an embedding into $U^3\oplus (-6)$ such that it will not contain any wall divisors. We remark that, as all wall divisors have nontrivial divisibility, a sufficient condition is that all elements of these lattices are embedded with trivial divisibility. Moreover, notice that all lattices of rank at most $3$ can be embedded into $U^3$ and, indeed, all these lattices correspond to induced automorphisms (cf. \cite[Sec. 4]{MTW15} and references therein for the corresponding automorphisms on abelian surfaces).\\
With the same techniques used in the non-symplectic case, we obtain the following list of groups that can act symplectically on generalised Kummer fourfolds:

\begin{tabular}{|c|c|c|c|c|}\hline
Rank & $G$ & $S_G(X)$ & $T_G(X)$ & Is induced?\\\hline
2 & $2$ & $A_1^2(-1)$ & $U\oplus A_1^2\oplus (-6)$ & Yes\\\hline
2 & $3$ & $A_2(-1)$ & $U\oplus A_2\oplus (-6)$ & Yes\\\hline
3 & $2^2$ & $A_1^3(-1)$ & $(-6)\oplus A_1^3$ & Yes\\\hline
3 & $\mathfrak{S}_3$ & $A_1(-1)\oplus A_2(-1)$ & $(-6)\oplus A_1\oplus A_2$ & Yes\\\hline
3 & $\mathfrak{A}_4$ & $A_3(-1)$ & $(-6)\oplus A_3$ & Yes\\\hline
4 & $2.\mathfrak{S}_3$ & $A_1^2(-1)\oplus A_2(-1)$ &  $\left( \begin{array}{ccc} 4 & -2 & 0\\ -2 & 4 & 0\\ 0 & 0 & 6 \end{array} \right)$ & No\\\hline
4 & $\mathfrak{S}_4$ & $A_1(-1)\oplus A_3(-1)$ & $A_1^2\oplus (12)$ & No\\\hline
4 & $\mathfrak{A}_5$ & $A_4(-1)$ & $A_1\oplus \left( \begin{array}{cc} 2 & -1\\ -1 & 8 \end{array} \right)$ & No\\\hline
4 & $\mathfrak{A}_5$ & $A_4(-1)$ & $A_2\oplus (10)$ & No\\\hline
4 & $2^3.\mathfrak{A}_4$ & $D_4(-1)$ & $\left( \begin{array}{ccc} 4 & -2 & 6\\ -2 & 4 & 0\\ 6 & 0 & 14 \end{array} \right)$ & No\\\hline
\end{tabular}

When $T_G$ is positive definite, its uniqueness was checked with \cite{BI58}. 

\vspace{10pt}

\section{Higher dimensions}
Let us briefly discuss what happens for generalised Kummer manifolds in higher dimensions. The techniques used in the previous sections can be used for arbitrary dimensions. More specifically, to classify non-symplectic automorphisms as in \Ref{sec}{nonsymp4}, one needs to find primitive embeddings of an element of square $2n+2$ inside $T_G(\Lambda)$ (or $S_G(\Lambda)$ in some cases of order two) and then classify its isometry orbit in this lattice. As this task is computationally cumbersome, we limit ourselves to specify, for every action of $G$ on $\Lambda$, what is the smallest dimension where we have a specific $G$ action on a manifold of Kummer type. We omit all groups already arising in the case of fourfolds. Note that this includes all groups of \Ref{ssec}{2var}. The numbering in the table refers to the subsection of \Ref{sec}{nonsympl} where the lattices appear.

\vspace{5pt}
\begin{tabular}{|c|c|c|c|c|}\hline
No. & Order & $S_G(\Lambda)$ & $T_G(\Lambda)$  & Minimal dimension \\\hline\hline
$1.1$ & 2 & $U^{\oplus2}\oplus \lat{-2}^{\oplus2}$ & $\lat{2}^{\oplus2} $ & 8 \\\hline
$1.5$ & 2 & $U(2)\oplus U(2)$ & $U(2)\oplus U(2)$ & 6 \\\hline
$5.1$ & 7 & $U^{\oplus 2}\oplus K_7(-1)$ & $K_7$ & 6 \\\hline
\end{tabular}
\vspace{5pt}

In the symplectic case, one needs to determine which are all possible wall divisors in the desired dimension. This can be computed using \cite[Prop. 1.3]{Yos12} and in particular we have that every wall divisor has nontrivial divisibility. In the list of \Ref{sec}{sympl}, only two possible actions on $E_8(-1)$ did not occur as automorphisms of Kummer fourfolds, namely the coinvariant lattice $A_1^4(-1)$ for an involution and the coinvariant lattice $A_2^2(-1)$ for an order three automorphism. We have the following two results:
\begin{prop}
There is no symplectic involution on a generalised Kummer manifold with coinvariant lattice $A_1^4(-1)$.
\end{prop}
\begin{proof}
Any primitive element of $A_1^4$ has divisibility two. Let $v$ be an element of square $2n+2$ and let $w$ be the element of the orthogonal $A_1^4(-1)$ with square $-2n-2$ and such that $\frac{v+w}{2}\in\Lambda$. The choice of $v$ defines an embedding of $A_1^4(-1)$ in the Kummer $n$ lattice. Let $X$ be a Kummer $n$ manifold such that $f^{-1}(A_1^4(-1))\subset Pic(X)$ for some marking $f$. By \cite[Prop 1.3]{Yos12}, $f^{-1}(w)$ is a wall divisor on $X$, therefore the involution of $A_1^4(-1)$ cannot be induced by a regular involution. This holds for any choice of $v$, therefore our claim holds.
\end{proof}

\begin{prop}
Suppose that $n+1$ is not divisible by 3. Then there are Kummer $2n$ folds with a symplectic order three automorphism with coinvariant lattice $A_2^2(-1)$.
\end{prop}
\begin{proof}
To prove our claim it is enough to find an element $v$ of square $2n+2$ in the unimodular complementary lattice $A_2^2$ such that $\langle v, A_2^2(-1)\rangle$ is saturated. This is equivalent to saying that $v$ defines an embedding of $A_2^2(-1)$ in the Kummer $n$ lattice where every element has trivial divisibility, so that the higher dimensional equivalent of \Ref{prop}{e8tosympl} applies. This is easily done by taking $v=t+s$, where $s$ is an element of the first copy of $A_2$ whose square is congruent to $2$ modulo $6$ (this can be done for any value of the square) and $t$ is congruent either to $0$ or to $2$ modulo $6$. Such an element has divisibility one, therefore $\langle v, A_2^2(-1)\rangle$ is saturated. 
\end{proof}
The previous two propositions allow us to state the following:
\begin{thm}
Let $G\subset SO(E_8(-1))$. Then there exists $n$ and a manifold $X$ of Kummer $n$ type such that $G\in O(H^2(X))$ is induced by symplectic automorphisms and $S_G(X)\cong S_G(E_8(-1))$ if and only if $rk(S_G(E_8(-1)))+l(A_{S_G(E_8(-1))})<8$.
\end{thm}

This is in analogy with a phenomenon in the case of manifolds of \kntipo , where it is conjectured that the same happens for all coinvariant lattices such that their rank and the length of their discriminant groups sum to the rank of the Mukai lattice (cf. \cite[Conj. 32]{Mon14}).

\end{document}